\documentclass[11 pt, reqno]{amsart}
\usepackage{amsfonts,amssymb,amscd,amsmath,enumerate,verbatim,calc}
\usepackage{amsfonts,amssymb,amscd,amsmath,enumerate,verbatim,calc}
\usepackage{url}
\theoremstyle{plain}
\newtheorem{theorem}{Theorem}
\newtheorem{general}{General Theorem}

\newtheorem{lemma}{Lemma}
\newtheorem{proposition}{Proposition}
\theoremstyle{definition}

\theoremstyle{remark}
\newtheorem{remark}{Remark}
\newtheorem{case}{Case}
\numberwithin{equation}{section}

\begin{document}

\title{On the Structure of the Continued Fraction of $\sqrt{d}$}

\author {Amrik Singh Nimbran}
\address{B3-304, Palm Grove Heights, Ardee City, Gurugram, Haryana, India 122003.}
\email{amrikn622@gmail.com}


\begin{abstract}
We examine the structure of the continued fraction  of $\sqrt{d}$, and give formulae for the repeated partial quotients in its period.
\end{abstract}

\maketitle

Motivated by discussion in Kraitchik's book \cite[Ch.III]{kraitchik}, I examined the period of the continued fraction expansion of $\sqrt{d}$ with $d$ a non-square positive integer. The aim was to seek patterns in the structure of its period. My investigation resulted in some new formulas relating to the structure of $\sqrt{d}$'s period: its \emph{central term} when the period-length is even, and the \emph{repeated partial quotients}. The paper is primarily a discussion of the methods used in obtaining these formulas. We will use linear recurrences and $2\times 2$ \emph{unimodular} matrices whose determinant equals $\pm 1.$ A `cousin' of Chebyshev polynomial of the second kind will also be introduced in the paper.

\section{Introduction}

Irrational square roots entered into mathematics with the Pythagorean theorem which led to the discovery of the irrationality of $\sqrt{2}$. R. Bombelli (\emph{Algebra}, 1572) and P. Cataldi (\emph{Trattato del modo brevissimo di trouare la Radici quadra delli numeri}, 1613), both from Bologna (Italy), introduced continued fractions for square roots.

For a real number $\alpha$, an expression of the form
\[
\alpha = a_0+\cfrac{1}{a_1+ \cfrac{1}{a_2+ \cfrac{1}{a_3 +\dotsb}}}
\]
with $a_i \in \mathbb{N}$ for $i \ge 1,$ is called a \emph{simple continued fraction} (scf) of $\alpha$. It is generally denoted by
a space-saving symbolism: $\left[a_0; a_1, a_2, a_3, \dots \right].$ The integers $a_i$'s are called \emph{partial quotients}. The rational number represented by the truncated continued fraction $\left[a_0; a_1, a_2, a_3, \dots, a_n\right]$, having initial $n+1$ terms from $a_0$ up to and including $a_n$, is called the $n$th \emph{convergent} ($c_n$) of $\alpha$. (The calculation is carried out in reverse order beginning from tail.)

Define the sequences $\lbrace p_k\rbrace$, $\lbrace q_k \rbrace$:
\begin{gather}
p_{-2} =0, \quad p_{-1} =1, \quad p_{k} = a_{k}~ p_{k-1} + p_{k-2} \; \text{for} \; k\ge 0, \\
q_{-2} =1, \quad q_{-1} =0, \quad q_{k} = a_{k}~ q_{k-1} + q_{k-2} \; \text{for} \; k\ge 0.
\end{gather}
Then $\displaystyle [a_0; a_1, \dots, a_k]=\frac{p_k}{q_k}.$

If the quotients repeat from a point $r$ onward, i.e., $a_{m\ell+r+k} =a_{r+k}, \, m\in \mathbb{N}, \, 0\le k\le \ell$, with period's length $\ell$, then the scf is said to be \emph{periodic} and written as $\displaystyle \left[a_0; a_1, a_2, \dots a_{r-1}, \overline{a_r, a_{r+1}, a_{r+2}, \dots, a_{\ell+r-1}}\right].$ Euler \cite{euler71} proved that the value of every periodic scf is a quadratic irrational of the form $\frac{P +\sqrt{d}}{Q}$ with $P, Q \in \mathbb{Z}, ~ Q\ne 0$ and $d \in \mathbb{Z}^+$ not a perfect square. It is a solution of some quadratic equation having integral coefficients: $ax^2 +bx+c=0$ ($a\ne 0$) whose \emph{discriminant} $d:=b^2 -4ac$ is not a perfect square. In 1770, Lagrange proved the converse of Euler's theorem that each quadratic irrational has a periodic scf expansion. If the scf repeats from the beginning, $[ \overline{a_0, a_1, \dots, a_n} ]$, it is called \emph{purely} periodic. Galois proved in 1828 that the scf of $x=\frac{P +\sqrt{d}}{Q}$ is purely periodic if and only if $x > 1$ and its conjugate $\bar{x} =\frac{P -\sqrt{d}}{Q}$ satisfies $-1<\bar{x} <0$.

We describe here the continued fraction algorithm for $\sqrt{d}$ for the sake of completeness. (See \cite[Ch 8.4]{sierpinski} for more details.) To begin with, we put $a_0 =\left[\sqrt{d}\right]=P_1$, the greatest integer $< \sqrt{d}$, and $d-a_0^2 =Q_1$. Then
\[
\displaystyle \sqrt{d} = a_0 + (\sqrt{d} - a_0) =a_0 +\frac{{d -a_0^2}}{\sqrt{d} +a_0} = a_0 +\cfrac{1}{\frac{\sqrt{d} +P_1}{Q_1}}.
\]
Now add to and subtract from $\sqrt{d} +P_1$ a positive integer $P_2$, the greatest integer $< \sqrt{d}$, such that $Q_1$ divides $P_1 +P_2$ to get: $a_1=\frac{P_1 +P_2}{Q_1}$. So we have
\[
\frac{\sqrt{d} +P_1}{Q_1} = a_1 +\frac{\sqrt{d} -P_2}{Q_1} =a_1 +\frac{d -P^2_2}{Q_1 (\sqrt{d} +P_2)}; \quad \text{thus} \;
a_1 =\left[ \frac{1}{\sqrt{d} -a_0} \right].
\]
Again add to and subtract from $\sqrt{d} +P_2$ a positive integer $P_3$, the greatest integer $< \sqrt{d}$, such that $Q_2$ divides $P_2 +P_3$ to get: $a_2=\frac{P_2 +P_3}{Q_2}$. So we have
\[
\frac{\sqrt{d} +P_2}{Q_2} = a_2 +\frac{\sqrt{d} -P_3}{Q_2} =a_2 +\frac{d -P^2_3}{Q_2 (\sqrt{d} +P_3)}; \quad \text{thus} \;
a_2 =\left[ \frac{1}{\frac{1}{\sqrt{d} -a_0} - a_1} \right].
\]
We repeat the process until we get $Q_n =1$ for some $n$ when the partial quotients will begin to repeat themselves. Since $1\le P_i \le [\sqrt{d}]$, $P_i$ can take on only a limited number of integral values. Consequently, $Q_i$ too assumes the same number of associated values as we have: $Q_n Q_{n-1} =d -P_n^2; \, a_n Q_n =P_n +P_{n+1}.$ (See Theorem 8.1 in \cite[pp.262-263]{hua}.)

The period of the scf of $\sqrt{d}$, with $d$ not a square, is symmetrical and for period length $\ell$ its form is:
$
\sqrt{d} = \left[a_0;~ \overline{a_1, a_2, a_3, \dots, a_{\ell-3}, a_{\ell-2}, a_{\ell-1}, 2a_0} \right]
$
with $a_1 =a_{\ell-1}, \, a_2=a_{\ell -2}, \dots$ So, $a_1, a_2, \dots, a_{\ell-1}$ form a palindrome. The palindrome may or may not have a central term. When $\ell(d) =2n$, the period is symmetric around $a_n$ with $a_{n+1}=a_{n-1}, \, a_{n+2}=a_{n-2}$ and so on. If $\ell(d) =2n+1$, then $a_{n+1}=a_{n}, \, a_{n+2}=a_{n-1}$ and so on. $\ell (d)$ is odd if and only if $d=a^2+b^2, \, (a, b)=1$.\cite{arnold, rippon} In the scf of $\sqrt{d}$ with period length $\ell$, each partial quotient $a_k$ for $0 \le k < \ell$ satisfies $a_k < \sqrt{d}$. \cite[p.245, 3(f)]{LeVeque}

\section{Tools and preliminary results}

\subsection{Linear recurrences}

The general form of linear recurrences is: $u_{n+1} =a~ u_{n} +b~ u_{n-1}$ where $a \in \mathbb{Z}^{+}, \; b \in \mathbb{Z}\setminus \lbrace0\rbrace$, and two initial values $u_0, \, u_1$ are given. Its characteristic equation is: $x^2 -a x -b=0.$ If $\alpha >0, \, \beta <0$ are its roots, then $\alpha +\beta =a, \; \alpha \cdot \beta =-b, \; \alpha -\beta = \sqrt{a^2 +4b}.$ If two numbers $\lambda, \, \mu$ can be chosen which satisfy $\lambda +\mu =u_0; \; \alpha \lambda +\beta \mu =u_1,$ then $u_n = \lambda \alpha^n +\mu \beta^n$ for $n = 0, 1, 2, \dots$ \cite[p.199, Th.4.10, eq(4.5)]{niven}

The recurrences appearing in this paper mostly have $u_0=0, \, u_1=1$ and: (i). $a=2, b=1$; (ii). $a=4, b=-1$; (iii). $a=2m+1 \, (m\in \mathbb{N}_0), \, b=1$ and (iv). $a=m, \, b=1; \; a=4m, \, b=m \, (m\in \mathbb{N}).$

\subsection{Two useful expansions}

Consider the quadratic equation:
\begin{equation}
x^2 -(2m+1)x -1=0,                \label{quadeq}
\end{equation}
whose two roots are given by
\begin{equation}
x=\frac{2m+1 \pm \sqrt{(2m+1)^2 +4}}{2}.  \label{root}
\end{equation}

Now the equation \eqref{quadeq} can be written as
\begin{equation}
x = (2m+1) +\frac{1}{x}.     \label{quadeq2}
\end{equation}
Putting the RHS of \eqref{quadeq2} for $x$ repeatedly, and the positive root from \eqref{root} on the LHS of the preceding equation leads to:
\[
\frac{2m+1 +\sqrt{(2m+1)^2 +4}}{2} =2m+1+
\cfrac{1}{2m+1+
\cfrac{1}{2m+1+
\cfrac{1}{2m+1+
\dotsb
}}}
\]
which is usually denoted by $[2m+1; \overline{2m+1}]$. The convergents of this continued fraction are given by
\[
\frac{u_2}{u_1}, \, \frac{u_3}{u_2}, \dots, \frac{u_{n+1}}{u_n}, \dots
\]
where $u_n$ ($n=0, 1, 2, \dots)$ has the closed form:

\begin{equation}
u_n = \frac{\left(2m+1+\sqrt{(2m+1)^2+4}\right)^n -\left(2m+1 -\sqrt{(2m+1)^2+4} \right)^n}{2^n ~ \sqrt{(2m+1)^2+4}}. \label{power}
\end{equation}
Note that $u_{n}$ is even when $n=3r$, and odd when $n= 3r+1, \, 3r+2.$

$m=0$ yields a purely periodic scf for the \emph{golden ratio}:
\begin{equation*}
\varphi: = \frac{1 +\sqrt{5}}{2} = 1 +\cfrac{1}{1 +\cfrac{1}{1+\cfrac{1}{1+\dotsb}}}
\end{equation*}
$\varphi$'s convergents are:
\[
\frac{1}{1}, \frac{2}{1}, \frac{3}{2}, \frac{5}{3}, \frac{8}{5}, \frac{13}{8}, \frac{21}{13}, \dots, \frac{F_{n+1}}{F_n}, \dots
\]
where $F_n$ being the $n$-th the Fibonacci number. The Fibonacci sequence $\left(F_n\right)_{n\ge 0}$ is given by the recurrence relation: $\displaystyle F_{n+1} = F_{n} + F_{n-1}$ for $n\ge 1$ and $ F_{0}=0, \; F_{1}=1.$
With $ \displaystyle a=b=1, \alpha =\frac{1+\sqrt{5}}{2}, \, \beta = \frac{1-\sqrt{5}}{2}, \lambda =\frac{1}{\sqrt{5}}, \mu =-\frac{1}{\sqrt{5}}$, we get the formula given by Euler\cite[\S 5]{euler326} and later by Binet\cite[p.563]{binet}:
\[
F_n =\frac{1}{\sqrt{5}} \left[ \left(\frac{1 +\sqrt{5}}{2}\right)^n - \left(\frac{1 -\sqrt{5}}{2} \right)^n\right].
\]

Lucas\cite[(21)]{lucas} recorded this representation ($n\ge 1$):
\[
\frac{1}{2} \frac{(1 +\sqrt{5})^{n+1} -(1 -\sqrt{5})^{n+1}}{(1 +\sqrt{5})^{n} -(1 -\sqrt{5})^{n}}
= 1 +\cfrac{1}{1 +\cfrac{1}{1+\cfrac{1}{1+\dotsb}}}
\]

It is obvious that $x^2 -2mx -k =0$ ($m\in \mathbb{N}, \, 1\le k \le m$) yields the quadratic irrational $m \pm \sqrt{m^2+k}.$
The equation: $\displaystyle x^2 -2x-1=0$, having roots: $1 \pm \sqrt{2},$ can be rewritten as:
\[
x^2 = 2x +1 \Longrightarrow x = 2 +\frac{1}{x}.
\]
Substituting $2 +\frac{1}{x}$ for $x$ repeatedly on the RHS with the positive root on the LHS yields
\begin{equation*}
1 +\sqrt{2} = 2 +\cfrac{1}{2 +\cfrac{1}{2+\cfrac{1}{2+\dotsb}}} \Longrightarrow \sqrt{2} = 1 +\cfrac{1}{2 +\cfrac{1}{2+\cfrac{1}{2+\dotsb}}}
\end{equation*}
where the left scf is purely periodic unlike the right one. In general, the scf of $\sqrt{d}$ is not purely periodic while that of  $[\sqrt{d}] +\sqrt{d}$, where $[\sqrt{d}]$ denotes the greatest integer less than $\sqrt{d},$ is purely periodic.

The convergents of $\sqrt{2} =[1; \overline{2}]$ are given by $\displaystyle \frac{p_{n+1}}{q_{n+1}} =\frac{2p_{n} +p_{n-1}}{2q_{n} +q_{n-1}}$ for $n\ge 1, \, p_0=1, q_0=0, p_1=1, q_1=1$:
\[
\frac{1}{1}, \; \frac{3}{2}, \; \frac{7}{5}, \; \frac{17}{12}, \; \frac{41}{29}, \; \frac{99}{70}, \; \frac{239}{169}, \dots
\]
The numerators and denominators occur in \url{https://oeis.org/A001333} and \url{https://oeis.org/A000129} with the closed forms:
\[
p_{k}= \frac{(1+\sqrt{2})^{k} +(1 -\sqrt{2})^{k}}{2}; \quad q_{k}= \frac{(1+\sqrt{2})^{k} -(1 -\sqrt{2})^{k}}{2\sqrt{2}}.
\]
from these two closed forms we obtain
\begin{equation}
p_{k} +q_{k} = \frac{(1+\sqrt{2})^{k+1} -(1 -\sqrt{2})^{k+1}}{2\sqrt{2}} =q_{k+1}, \label{root2a}
\end{equation}
and
\begin{equation}
p_{k} +2q_{k} = \frac{(1+\sqrt{2})^{k+1} +(1 -\sqrt{2})^{k+1}}{2} =p_{k+1}. \label{root2b}
\end{equation}

\subsection{Matrix method in continued fractions}

\subsubsection{A correspondence between matrices and convergents}
Given a sequence $a_0, a_1, a_2, \dots,$ and defining $ A_k:= \begin{bmatrix}
a_{k} & 1\\
1 & 0
\end{bmatrix},$ we have this correspondence:
\begin{equation}
\prod_{k=0}^{n} A_k =
\begin{bmatrix}
p_{n} & p_{n-1}\\
q_{n} & q_{n-1}
 \end{bmatrix}
 \quad \text{for} \; n=0, 1, 2, \dots,
\label{R1}
\end{equation}
if and only if \; $\displaystyle \frac{p_n}{q_n} = [a_0; a_1, \dots, a_n]$ for $n=0, 1, 2, \dots .$ This sets up a correspondence between certain products of $2\times 2$ matrices and continued fractions. (see \cite[p.28]{borwein} \cite{frame} \cite[p.244]{hurwitz} \cite[p.99]{poorten85} \cite[p.87]{poorten90}). Note that all the matrices occurring in the product are unimodular with $\det \begin{vmatrix}
p_{n} & p_{n-1}\\
q_{n} & q_{n-1}
\end{vmatrix}
=p_{n} q_{n-1} -q_{n} p_{n-1}
=(-1)^{n-1}.$
Writing $\alpha =a_0 +\frac{1}{\alpha_1} \; (\alpha_1>1), \, \alpha_1 =a_1 +\frac{1}{\alpha_2} \; (\alpha_2>1)$ and $\alpha_n =a_n +\frac{1}{\alpha_{n+1}} \; (\alpha_{n+1}>1)$, Davenport \cite[80, eq(14)]{davenport} obtained his ``most serviceable formula" expressing $\alpha$ in terms of the \emph{complete quotient} $\alpha_{n+1}$ and the two convergents $\frac{A_n}{B_n}$ \& $\frac{A_{n-1}}{B_{n-1}}$
\begin{equation}
\alpha =\frac{\alpha_{n+1} A_n +A_{n-1}}{\alpha_{n+1} B_n +B_{n-1}}  \label{daven}
\end{equation}
Using \eqref{daven} for period length $\ell,$ Rippon and Taylor deduced for $\sqrt{d} =\sqrt{a^2 +b}$ their Lemma 1(a) which we will use in our proofs:
\begin{lemma}
\begin{equation}
a = a_0; \qquad b = \frac{2 a_0 B +C}{A} \label{rip}
\end{equation}
$A, B, C$ being the elements of
$
\begin{bmatrix}
A & B\\
B & C
\end{bmatrix}
=
\begin{bmatrix}
p_{\ell -1} & p_{\ell -2}\\
q_{\ell -1} & q_{\ell -2}
\end{bmatrix}.
$
\end{lemma}

Two irrational numbers $\alpha$ and $\beta$ are said to be \emph{equivalent} if $A \alpha =\beta$ for some $A =
\begin{bmatrix}
a_{11} & a_{12}\\
a_{21} & a_{22}
\end{bmatrix}$
with integer elements and $\det A=\pm 1.$ It is defined by the rule: $A \alpha = \displaystyle  \frac{a_{11} \alpha +a_{12}}{a_{21} \alpha + a_{22}}.$ \cite[p.37]{borwein}

\subsubsection{Powers of a matrix associated with convergents of $\sqrt{2}$}
As Khovanskii \cite[p.292, Ex.20]{khovanskii} states, the matrix
$\begin{bmatrix}
1 & 2\\
1 & 1
\end{bmatrix}
=
\begin{bmatrix}
p_1 & 2q_1\\
q_1 & p_1
\end{bmatrix}$
leads to $\sqrt{2}=[1; \overline{2}]$.

\begin{lemma}\label{lemma1}
\begin{equation}
\begin{bmatrix}
1 & 2\\
1 & 1
\end{bmatrix}^{k}
=
\begin{bmatrix}
p_k & 2q_k\\
q_k & p_k
\end{bmatrix}. \label{root2c}
\end{equation}
\end{lemma}
\begin{proof}
The statement is obviously true for $k=1.$ Assume it to be true $k=m$ so that
\[
\begin{bmatrix}
1 & 2\\
1 & 1
\end{bmatrix}^{m}
=
\begin{bmatrix}
p_m & 2q_m\\
q_m & p_m
\end{bmatrix}.
\]
Now
\[
\begin{bmatrix}
p_m & 2q_m\\
q_m & p_m
\end{bmatrix}.
\cdot
\begin{bmatrix}
1 & 2\\
1 & 1
\end{bmatrix}
=
\begin{bmatrix}
p_m +2q_m & 2p_m +2q_m\\
p_m +q_m & p_m +2q_m
\end{bmatrix}
=
\begin{bmatrix}
p_{m+1} & 2q_{m+1}\\
q_{m+1} & p_{m+1}
\end{bmatrix},
\]
by using \eqref{root2a} and \eqref{root2a}. Thus we get
\[
\begin{bmatrix}
1 & 2\\
1 & 1
\end{bmatrix}^{m+1}
=
\begin{bmatrix}
p_{m+1} & 2q_{m+1}\\
q_{m+1} & p_{m+1}
\end{bmatrix}
\]
The statement is thus true for $k=m+1$ also; hence for all $k.$
\end{proof}

\subsubsection{Matrix associated with $\sqrt{3}$}

The convergents $\displaystyle \frac{p_k}{q_k}$ of $\sqrt{3} =[1; \overline{1, 2}]$ are:
\[
\frac{2}{1}, \, \frac{5}{3}, \, \frac{7}{4}, \, \frac{19}{11}, \, \frac{26}{15}, \, \frac{71}{41}, \, \frac{97}{56}, \, \frac{265}{153}, \dots
\]
whose numerators occur appear in sequence A048788 at \url{https://oeis.org/A048788} and denominators in sequence A002530 at \url{https://oeis.org/A002530}. Note these relations:
\begin{align}
& p_{2n-1} =q_{2n} -q_{2n-1}; \; p_{2n+1} =q_{2n} +q_{2n+1};  \label{3a}\\
& 3q_{2n} +2q_{2n-1} =q_{2n+2}; \; 3q_{2n-1} +q_{2n-2} =q_{2n+1}. \label{3b}
\end{align}
The closed forms are:
\begin{align*}
& p_{2n-1} = \frac{(1+\sqrt{3})^{2n} +(1-\sqrt{3})^{2n} }{2^{n+1}}; \; q_{2n-1} = \frac{(1+\sqrt{3})^{2n} +(1-\sqrt{3})^{2n} }{2^{n+1}\sqrt{3}}, \\
& p_{2n} = \frac{(1+\sqrt{3})^{2n+1} +(1-\sqrt{3})^{2n+1} }{2^{n+1}}; \; q_{2n} = \frac{(1+\sqrt{3})^{2n+1} +(1-\sqrt{3})^{2n+1} }{2^{n+1}\sqrt{3}}.
\end{align*}
Convergents of $\displaystyle \frac{p'_k}{q'_k}$ of $1+\sqrt{3}$ can be obtained by adding 1 to convergents of $\sqrt{3}$.
By using the relations \eqref{3a} and \eqref{3b}, we can prove by induction:
\begin{lemma}\label{lemma2}
\begin{equation}
\begin{bmatrix}
3 & 2\\
1 & 1
\end{bmatrix}^k
=
\begin{bmatrix}
q_{2k} & 2q_{2k-1}\\
q_{2k-1} & q_{2k-2}
\end{bmatrix}
=
\begin{bmatrix}
p'_{2k-1} & 2q'_{2k-1}\\
q'_{2k-1} & p'_{2k-3}
\end{bmatrix},
\qquad k\ge 1.
\label{eq:l2}
\end{equation}
\end{lemma}

\subsubsection{Powers of a general matrix}
We take up a matrix to be used later.
\begin{lemma}\label{lemma3}
\begin{equation}
\begin{bmatrix}
2m+1 & 1 \\
1 & 0
\end{bmatrix}^{n}
=
\begin{bmatrix}
u_{n+1} & u_{n} \\
u_{n} & u_{n-1}
\end{bmatrix},
\end{equation}
where $u_n$ satisfies the recurrence relation:
\begin{equation}
u_{n+1}=(2m+1)~ u_{n} +u_{n-1}; \; \text{with} \; u_1=1, \, u_{0} =0. \label{recurrence}
\end{equation}
\end{lemma}

\begin{proof}
The statement is obviously true for $n = 1$ in view of \eqref{recurrence}.  Assume the statement to be true for $n=m$ so that
\[
\begin{bmatrix}
2m+1 & 1  \\
1 & 0
\end{bmatrix}^m
=
\begin{bmatrix}
u_{m+1} & u_{m}  \\
u_{m} & u_{m-1}
\end{bmatrix}.
\]
Then
\begin{align*}
\begin{bmatrix}
2m+1 & 1  \\
1 & 0
\end{bmatrix}^{m+1}
&=
\begin{bmatrix}
u_{m+1} & u_{m}  \\
u_{m} & u_{m-1}
\end{bmatrix}
\begin{bmatrix}
2m+1 & 1  \\
1 & 0
\end{bmatrix}\\
& =
\begin{bmatrix}
(2m+1) u_{m+1} +u_{m} &  u_{m+1}  \\
(2m+1) u_{m} +u_{m-1} &  u_{m}
\end{bmatrix} \\
& =
\begin{bmatrix}
u_{m+2} & u_{m+1}  \\
u_{m+1} & u_{m}
\end{bmatrix}
\qquad \text{by using \eqref{recurrence}}.
\end{align*}
So the statement is true for $n=m+1$ also. Hence true for all $n$.
\end{proof}
\subsubsection{Proof using Chebyshev polynomials}
Let $M = \begin{bmatrix}
m_{11} & m_{12}\\
m_{21} & m_{22}
\end{bmatrix}$
be a unimodular matrix. Then \cite[pp.70--71]{bornwolf}
\begin{equation}
M^{n}
=\begin{bmatrix}
m_{11}~ U_{n-1}(a)-U_{n-2}(a) & m_{12}~ U_{n-1}(a)\\
m_{21}~ U_{n-1}(a) & m_{22}~ U_{n-1}(a) - U_{n-2}(a)
\end{bmatrix}
\end{equation}
where $\displaystyle a = \frac{m_{11} +m_{22}}{2}$ and $\displaystyle U_n$ are the \emph{Chebyshev polynomials of the second kind} whose generating function is: $\displaystyle g(x, t) = \frac{1}{1-2xt+t^2} = \sum_{n=0}^{\infty} U_{n} (x)~t^n$ with $|t|<1, \, |x|\le 1.$ We have $U_{n} (x)$  \cite[p.229, (5.98)]{andrews}:
\[
U_{n} (x) = \sum_{k=0}^{[\frac{n}{2}]} (-1)^k \binom{n-k}{k}(2x)^{n-2k}.
\]
A first few polynomials are given below.
\begin{align*}
& U_{0} (x) =1, \; U_{1} (x) =2x, \; U_{2} (x) =(2x)^2 -1, \; U_{3} (x) =(2x)^3 -2(2x), \\
& U_{4} (x) =(2x)^4 -3(2x)^2 +1, \;  U_{5} (x) =(2x)^5 -4 (2x)^3 +3(2x).
\end{align*}
They satisfy the recurrence relation: $U_{n+1} (x) = 2x~ U_{n} (x) -U_{n-1} (x).$  \cite[p.230, (5.100)]{andrews} Note that $a=2x, \, b=-1$ in the recurrence.

Careful examination of the polynomial with all $[\frac{n}{2}]+1$ positive terms (in descending order of the powers of $x$) generated by
\begin{equation}
U^\prime_{n} (x):=\frac{\left(x +\sqrt{1+x^2}\right)^{n+1} -\left(x -\sqrt{1+x^2}\right)^{n+1}}{2 ~ \sqrt{1+x^2}}, \; n=0, 1, 2, \dots \label{binetlike}
\end{equation}
shows that the introduction of the sign scheme ($+ ~-$) in $U^\prime_{n} (x)$ transforms it into $U_{n} (x)$. $U^\prime_{n} (x)$ comes from $t^2 \rightarrow -t^2$ in the generating function.

With $m_{11} =2m+1, \, m_{22}=0 \Rightarrow \displaystyle a=\frac{2m+1}{2}$,  we get:
\[
\begin{bmatrix}
2m+1 & 1\\
1 & 0
\end{bmatrix}^n
=
\begin{bmatrix}
(2m+1) ~ U_{n-1}(\frac{2m+1}{2}) -U_{n-2}(\frac{2m+1}{2}) & U_{n-1}(\frac{2m+1}{2})\\
U_{n-1}(\frac{2m+1}{2}) & -U_{n-2}(\frac{2m+1}{2})
\end{bmatrix}
\]
and on using the recurrence relation for these polynomials:
\begin{equation}
\begin{bmatrix}
2m+1 & 1\\
1 & 0
\end{bmatrix}^n
=
\begin{bmatrix}
U_{n}(\frac{2m+1}{2}) & U_{n-1}(\frac{2m+1}{2})\\
U_{n-1}(\frac{2m+1}{2}) & -U_{n-2}(\frac{2m+1}{2})
\end{bmatrix}
=
\begin{bmatrix}
u_{n+1} & u_{n}\\
u_{n} & u_{n-1}
\end{bmatrix}
\end{equation}
changing the subscript from $n$ to $n+1$ as $U_0 =u_1=1$. $u_n$ is obtained by setting $x=\frac{2m+1}{2}$ in \eqref{binetlike} and was given earlier in \eqref{power}.

\section{Formulas without a central term or the central term $< a_0$}

\subsection{Formulas without central term or central term $< a_0 -1$}

\subsubsection{Formulas with $\ell (d)=1$}

The only formula was noted by Euler:
\begin{equation}
\sqrt{n^2 +1}=[n; \overline{2n}]; \quad n\ge 1.
\end{equation}

\subsubsection{Formulas with $\ell (d)=2$}

These  two formulas can be easily established.
\begin{equation}
\sqrt{(mn)^2 +n}=[mn; \overline{2m, 2mn}]; \quad m, n \ge 1.
\end{equation}
\begin{equation}
\sqrt{(mn)^2 +2n}=[mn; \overline{m, 2mn}]; \quad m, n \ge 1. \label{amrik1}
\end{equation}
It can be shown without difficulty that these are the only possibilities.

\subsubsection{Formula with $\ell (d)=3$}
For the most general case with period 3, we have $\sqrt {d} = [a; b, b, 2a, b, b,2a, \dots].$ Then $\sqrt{d} =a+y$ with $y =[0;b, b, 2a,b, b, 2a, \dots].$ So we get
\[
y = \cfrac{1}{b+ \cfrac{1} {b+ \cfrac{1}{2a +y}}} =\frac{2ab+by +1}{2ab^2 +2a +by^2 +b+y}
\]
or
\[
y^2+2a y-\frac{2ab+1}{b^2+1}=0.
\]
Solving the resultant quadratic equation, we get the positive root: $y \displaystyle =-a +\sqrt{a^2 +\frac{2ab+1}{b^2 +1}}$. Combining with the leading $a$, we obtain $d \displaystyle = a^2 + \frac{2ab+1}{b^2+1}$.  For $d$ to be an integer, $b^2+1$ must divide $2ab+1$. Let $\displaystyle \frac{2ab+1}{b^2+1} =k.$ \emph{WolframAlpha} gives its general solution over the integers:
\[
b=2c_1, ~ a=4c_2c_1^2+c_1+c_2, ~  k=4c_1c_2 +1; \quad c_1, c_2 \in \mathbb{Z}.
\]
This leads, on setting $c_1=m, ~c_2=n,$ to Perron's formula \cite[p.100]{perron}:
\begin{align}
& \sqrt{\lbrace (4m^2+1) n +m \rbrace^2 +4mn +1}=[(4m^2+1)n +m; \notag \\
& \overline{2m, 2m, 2\lbrace (4m^2+1) n +m \rbrace}].
\end{align}

\subsubsection{Formula with $\ell (d)=5$}

This formula will be used in a later section:
\begin{equation}
\sqrt{(2n+1)^2 +4} =[2n+1; ~ \overline{n, 1, 1, n, 4n+2}]. \label{amrik5a}
\end{equation}
The same procedure applied to $y =[0; \frac{a-1}{2}, 1, 1,  \frac{a-1}{2}, 2a+y$ gives equation $y^2 +2ay -4=0$ which yields $y=-a \pm \sqrt{a^2 +4}$ establishing the formula.
\begin{proof} We prove it by the continued fraction algorithm outlined above.
\[
\sqrt{(2n+1)^2 +4} =  (2n +1) + \frac{1}{\displaystyle \frac{ \sqrt {(2n+1)^2 +4} +(2n+1) }{4}}; \;  a_0=2n+1.
\]
Next
\[
\frac{ \sqrt {(2n+1)^2 +4} +(2n+1) }{4} = n +  \frac{ \sqrt {(2n+1)^2 +4} -(2n-1) }{4}.
\]
So $a_1=n$ and the expression on the extreme right becomes
\[
 \frac{(2n+1)^2 +4 -(2n-1)^2 }{4[ \sqrt {(2n+1)^2 +4} +(2n-1) ]}
 = \frac{2n+1}{\sqrt {(2n+1)^2 +4} +(2n-1)}
\]
which on being inverted becomes
\[
 \frac{\sqrt {(2n+1)^2 +4} +(2n-1)} {2n+1} = 1 +  \frac{\sqrt {(2n+1)^2 +4} -2} {2n+1}.
\]
So $a_2=1$ and the expression on the extreme right becomes
\[
 \frac{(2n+1)^2 } {(2n+1) [\sqrt {(2n+1)^2 +4} +2)] } = \frac{2n+1}{\sqrt {(2n+1)^2 +4} +2}
\]
which on being inverted becomes
\[
 \frac{\sqrt {(2n+1)^2 +4} +2 +2n -2n} {2n+1} = 1 +  \frac{\sqrt {(2n+1)^2 +4} -(2n-1)} {2n+1}.
\]
So $a_3 =1$ and the expression on the extreme right becomes
\[
\frac{4}{\sqrt {(2n+1)^2 +4} +(2n-1)}
\]
which on being inverted becomes
\[
 \frac{\sqrt {(2n+1)^2 +4} +(2n-1)} {4} = n +  \frac{\sqrt {(2n+1)^2 +4} -(2n+1)} {4}.
\]
So $a_4=n$ and the expression on the extreme right becomes
\[
\frac{1}{\sqrt {(2n+1)^2 +4} +(2n+1)} = \frac{1}{(4n+2) +\sqrt {(2n+1)^2 +4} -(2n+1)}.
\]
Thus $a_5 =4n+2$ and the loop starts again.
\end{proof}

\begin{proposition}
Let $1 \le a_n < a_0 -1$ be a central term of the period of the scf of $\sqrt{d} = \sqrt{a^2+b}$. Then
\begin{align*}
a_n \equiv 1 \pmod 2 \Longleftrightarrow a \equiv b  \equiv 0 \pmod 2, \\
a_n \equiv 0 \pmod 2 \Longleftrightarrow a \equiv b  \equiv 1 \pmod 2.
\end{align*}
\end{proposition}

\subsubsection{Formula with $\ell (d)=6$}

This formula is provable by the algorithm:
\begin{equation}
\sqrt{(2n+2)^2 +(4n+1)} =[2n+2; ~ \overline{1, n, 2, n, 1, 4n+4}]; \quad n \in \mathbb{N}. \label{amrik6a}
\end{equation}

We now list formulas with particular values for $n=0, 1, 2, \dots$ (except the second which holds for $n\ge 1$):
\begin{align*}
\sqrt{(45n+14)^2 +(38n+12)} &=[45n+14; ~ \overline{2, 2, {\bf 1}, 2, 2, 2(45n+14)}], \\
\sqrt{(35n+1)^2 +(29n +1)} &=\left[35n+1; \overline{2, 2, 2, {\bf 2}, 2, 2(35n+1)} \right], \\
\sqrt{(95 n + 68)^2 +(78 n + 56)} &=[95 n + 68; ~ \overline{2, 2, {\bf 3}, 2, 2, 2(95 n + 68)}], \\
\sqrt{(60 n + 50 )^2 +(49 n + 41)} &=[60 n + 50; ~ \overline{2, 2, {\bf 4}, 2, 2, 2(60 n + 50)}], \\
\sqrt{(145 n + 17 )^2 +( 2 (59 n + 7)) } &=[145 n + 17; ~ \overline{2, 2, {\bf 5}, 2, 2, 2(145 n + 17)}], \\
\sqrt{(85 n + 54)^2 + ( 69 n + 44) } &=[85 n + 54; ~ \overline{2, 2, {\bf 6}, 2, 2, 2(85 n + 54)}], \\
\sqrt{(195 n + 101)^2 + (158 n + 82)} &=[195 n + 101; ~ \overline{2, 2, {\bf 7}, 2, 2, 2(195 n + 101)}], \\
\sqrt{(110 n + 48)^2 + (89 n + 39)} &=[110 n + 48; ~ \overline{2, 2, {\bf 8}, 2, 2, 2(110 n + 48)}].
\end{align*}

When the middle term is odd, the coefficient of $n$ in square quantity goes up by 50 and that of $n$ in linear quantity rises by 40 when we step up to the next odd number. If the middle term is even, the first term increases by 25 and the second by 20 while going to the next even number.

\subsubsection{Formulas with $\ell (d)=7$}
We have for $n=0, 1, 2, \dots,$ these formulae
\begin{align*}
\sqrt{(13n+7)^2 +(16n+8)} &=[13n+7; ~ \overline{1, 1, 1, 1, 1, 1, 2(13n+7)}], \\
\sqrt{(29n+27)^2 +(34n+32)} &=[29n+27; ~ \overline{1, 1, 2, 2, 1, 1, 2(29n+27)}], \\
\sqrt{(53 n + 43)^2 +(60 n + 49)} &=[53 n + 43; ~ \overline{1, 1, 3, 3, 1, 1, 2(53 n + 43)}], \\
\sqrt{(85 n + 16)^2 +(94 n + 18)} &=[85 n + 16; ~ \overline{1, 1, 4, 4, 1, 1, 2(85 n + 16)}], \\
\sqrt{(125 n + 8)^2 +(136 n + 9)} &=[125 n + 8; ~ \overline{1, 1, 5, 5, 1, 1, 2(125 n + 8)}].
\end{align*}
We see that the coefficient of the $n$ term in the square part equals $(2m+1)^2 +4$, if we denote the two identical middle terms by $m$, and the coefficient of $n$ in the linear term is $(2m+1)^2 +2m+5.$ However, the constant terms do not exhibit any pattern.

The following formulas form another group.
\begin{align*}
\sqrt{(25n+17)^2 +(36n+25)} &=[25n+17; ~ \overline{1, 2, 1, 1, 2, 1, 2(25n+17)}],\\
\sqrt{(109 n + 52)^2 +(152 n + 73)} &=[109 n + 52; ~ \overline{1, 2, 3, 3, 2, 1, 2(109 n + 52)}], \\
\sqrt{\lbrace5(53 n +49)\rbrace^2 +(364 n + 337)} &=[5(53 n +49); ~ \overline{1, 2, 5, 5, 2, 1, 10(53 n +49)}],\\
\sqrt{(493 n + 18)^2 +( 672 n + 25)} &=[493 n + 18; ~ \overline{1, 2, 7, 7, 2, 1, 2(493 n + 18)}].
\end{align*}
\textit{No formula exists with pattern}: $1, 2, 2m, 2m, 2, 1$.

\subsubsection{Formula with $\ell (d)=8$}
This formula with three general terms can be proved by the algorithm:
\begin{equation}
\sqrt{(4n+5)^2 +(8n+3)} =[4n+5; ~ \overline{1, n, 2, 2n+2, 2, n, 1, 2(4n+5)}]; \; n \in \mathbb{N}. \label{amrik8}
\end{equation}

\subsubsection{Formula with $\ell (d)=9$}

The following formulas with particular values are valid for $n=0, 1, 2, \dots$
\begin{align*}
\sqrt{(73n+10)^2 +(92n+13)} &=[73n+10; ~ \overline{1, 1, 1, 2, 2, 1, 1, 1, 2(73n+10)}], \\
\sqrt{(509 n + 321)^2 + (602 n + 380)} &=[509 n + 321; ~ \overline{1, 1, 2, 4, 4, 2, 1, 1, 2(509 n + 321)}], \\
\sqrt{(1985 n + 254)^2 + ( 2256 n + 289)} &=[1985 n + 254; ~ \overline{1, 1, 3, 6, 6, 3, 1, 1, 2(1985 n + 254)}].
\end{align*}
We also have:
\[
\sqrt{(1949 n + 201)^2 +(720 n + 281)} =[1949 n + 201; ~ \overline{1, 2, 3, 4, 4, 3, 2, 1, 2(1949 n + 201)}].
\]

\subsubsection{Formula with $\ell (d)=10$}
We have ($n=0, 1, 2, \dots$) this nice pattern:
\begin{align*}
& \sqrt{\lbrace5 (2021 n + 965)\rbrace^2 +(14102 n + 6734)} \\
&=[5 (2021 n + 965); ~ \overline{1, 2, 3, 4, 5, 4, 3, 2, 1, 10(2021 n + 965)}].
\end{align*}

\subsection{Formulas with central term $=a_{0} -1$}

\subsubsection{Formula with $\ell (d)=4$}

It is easy to establish by the algorithm that \cite[p.110]{perron}:
\begin{equation}
\sqrt{n^2 +(2n-1)}=[n; \overline{1, n-1, 1, 2n}], n\ge 2.
\end{equation}

For period length $\ell(d) >4,$ we have
\begin{proposition}
Let $a_n =a_0-1$ be the central term of the period of the scf of $\sqrt{d} = \sqrt{a^2+b}$. Then $a \equiv 1 \pmod 2  \Longleftrightarrow b \equiv 1 \pmod 4$ and $a \equiv 0 \pmod 2  \Longleftrightarrow b \equiv 3 \pmod 4$.
\end{proposition}
It implies that $b$ is always odd.

\subsubsection{Formula with $\ell (d)=10$}

Kraitchik \cite[p.49]{kraitchik} gives for $n=0, 1, 2, \dots$
\begin{equation*}
\sqrt{(9n+6)^2 +(10n+7)} =[9n+6; ~ \overline{1, 1, 3, 1, 9n+5, 1, 3, 1, 1, 18n+12}].
\end{equation*}
We have:
\begin{equation}
\sqrt{(11 n + 9)^2 +(6n+5)} =[11 n + 9; ~ \overline{3, 1, 1, 1, 11 n + 8, 1, 1, 1, 3, 2(11 n + 9)}].
\end{equation}

\begin{equation}
\sqrt{(27 n + 8)^2 +(10 n + 3)} =[27 n + 8; ~ \overline{5, 2, 1, 1, 27 n + 7, 1, 1, 2, 5, 2(27 n + 8)}].
\end{equation}

\subsubsection{Formula with $\ell (d)=12$}
I discovered, for $n=0, 1, 2, \dots$, this formula:
\begin{align}
& \sqrt{(47n +10)\rbrace^2 +(14n +3)} \notag\\
&=[47n +10; ~ \overline{6, 1, 2, 1, 1, 47n+9, 1, 1, 2, 1, 6, 2(47n +10)}].
\end{align}

\subsubsection{Formula with $\ell (d)=14$}
I also discovered, for $n=0, 1, 2, \dots$,
\begin{align}
& \sqrt{  (33 n + 11)^2 +(38 n + 13 )  } \notag\\
&=[(33 n + 11; ~ \overline{1, 1, 2, 1, 3, 1, (33 n+10, 1, 3, 1, 2, 1, 1, 2((33 n + 11)}].
\end{align}

\subsubsection{Formula with $\ell (d)=16$}
Further, I discovered, for $n=0, 1, 2, \dots$,
\begin{align}
& \sqrt{ (151 n + 9)^2 + (210 n + 13) } \notag\\
&=[151 n + 9; ~ \overline{1, 2, 3, 1, 1, 5, 1, 151n +8, 1, 5, 1, 1, 3, 2, 1 2(151 n + 9)}].
\end{align}

\subsubsection{Formula with $\ell (d)=18$}
Furthermore, we have for $n=0, 1, 2, \dots$,
\begin{align}
& \sqrt{ (627 n + 12)^2 +(962 n + 19) } \notag\\
&=[627 n + 12; ~ \overline{1, 3, 3, 2, 1, 1, 7, 1, 627n +11, 1, 7, 1, 1, 2, 3, 3, 1, 2(627 n + 12)}].
\end{align}

\subsubsection{Formula with $\ell (d)=20$}
Furthermore, we have for $n=0, 1, 2, \dots$,
\begin{align}
& \sqrt{(3383n +12)^2 + (1950n +7)} \notag\\
&=[3383n +12; ~ \overline{3, 2, 7, 1, 3, 4, 1, 1, 1, 3383n +11, 1, 1, 1, 4, 3, 1, 7, 2, 3, 2(3383n +12)}],\\
& \sqrt{(9041n +14)^2 + (1930 n + 3) } \notag\\
&=[9041n +14; ~ \overline{9, 2, 1, 2, 2, 5, 4, 1, 1, 9041n +13, 1, 1, 4, 5, 2, 2, 1, 2, 9, 2(9041n +14)}].
\end{align}

\subsection{Formula with repeated 2's}

Perron gives this formula \cite[p.114]{perron}:
\begin{equation*}
\sqrt{(3n+1)^2 +(2n+1)} =[3n+1; ~ \overline{2, 1, 3n, 1, 2, 6n+2}].
\end{equation*}
It can be verified that no formula exists if 1 is replaced with any other number in the pattern. Kraitchik \cite[p.47]{kraitchik} gives this formula:
\begin{equation*}
\sqrt{(7n+1)^2 +(6n+1)} =[7n+1; ~ \overline{2, 2, 1, 7n, 1, 2, 2, 14n+2}].
\end{equation*}
I found the continuation of the previous two formulas:
\begin{equation*}
\sqrt{(17n+1)^2 +(14+1)} =[34n+1; ~ \overline{2, 2, 2, 1, 17n, 1, 2, 2, 2, 34n+2}],
\end{equation*}
and
\begin{align*}
& \sqrt{(41n+1)^2 +(34n+1)} =\left [41n+1; \right. \notag \\
& \left. \overline{2, 2, 2, 2, 1, 41n, 1, 2, 2, 2, 2, 82n+2} \right],
\end{align*}
both of which can be easily proved by means of the algorithm.  In fact, these four formulas are special cases of the theorem:

\begin{theorem}
Let $p_{k}$ is the numerator in the $k$-th convergent of $\sqrt{2}$. Then
\begin{align*}
& \sqrt{(p_{k+1}~ n +1)^2 +(2 p_{k} ~n +1)} = \left[(p_{k} n +1); \right. \\
& \left. \overline{2~ \text{repeated} ~k ~\text{times}, 1, (p_{k+1} n), 1, 2~ \text{repeated} ~k~ \text{times},  2(p_{k+1} n +1)} \right]. \notag
\label{amrikg}
\end{align*}
\end{theorem}

\begin{proof}
Let us recall the formula \eqref{rip} and Lemma \ref{lemma1}. The first few powers of the matrix associated with the scf of $\sqrt{2}$ follow.
\[
\begin{bmatrix}
1 & 2\\
1 & 1
\end{bmatrix}^2
=\begin{bmatrix}
3 & 4\\
2 & 3
\end{bmatrix};
\quad
\begin{bmatrix}
1 & 2\\
1 & 1
\end{bmatrix}^3
=\begin{bmatrix}
7 & 10\\
5 & 7
\end{bmatrix}
\]
and
\[
\begin{bmatrix}
1 & 2\\
1 & 1
\end{bmatrix}^4
=\begin{bmatrix}
17 & 24\\
12 & 17
\end{bmatrix};
\quad
\begin{bmatrix}
1 & 2\\
1 & 1
\end{bmatrix}^5
=
\begin{bmatrix}
41 & 58 \\
29 & 41
\end{bmatrix}.
\]
The $(k+1)$th $\&$ $k$th power matrices jointly give a formula with 2 repeated $k$ times. Let us now apply the procedure used in deriving the scf for period 3 in subsection 3.3.3. Let $\sqrt{d} = \sqrt{a^2 +b} = a+ (\sqrt{a^2 +b} -a)$. Put $y =\sqrt{a^2 +b} -a$.

If $\sqrt{d} =[a; \overline {2, 1, a-1, 1, 2, 2a}].$ Then
\[
y = \cfrac{1}{2+ \cfrac{1}{1+ \cfrac {1}{a-1 +\cfrac{1} {2 +\cfrac{1}{1+ \cfrac{1}{y+2a}}}}}}
\]
after some long calculation leads to $3 y^2 + 6a y - (2a+1) =0.$ Solving the quadratic equation, we get the positive root: $y = -a +\sqrt{a^2 +\frac{2a+1}{3}}$. This implies that $b=\frac{2a+1}{3}$ which must be an integer if $d$ is to be an integer. The solution $a=3n+1$ gives $b=2n+1$.

When $\sqrt{d} =[a; \overline {2, 2, 1, a-1, 1, 2, 2, 2a}]$, the same procedure gives the equation: $7y^2 +14ay -(6a+1) =0$ whose positive root yields $b=\frac{6a+1}{7}$ and for $d$ to be an integer we have $a=7n+1$ and $b=6n+1.$

With $\sqrt{d} =[a; \overline {2, 2, 2, 1, a-1, 1, 2, 2, 2, 2a}]$ the procedure gives the equation: $17y^2 +34ay -(14a+3) =0$ whose positive root yields $b=\frac{14a+3}{17}$ and for $d$ to be an integer we have $a=17n+1$ and $b=14n+1.$

And if 2 is repeated $k$ times in the period, using \eqref{rip} we get the equation: $p_{k+1} y^2 +2p_{k+1} ay -(2p_{k} a+ p_{k-1}) =0$ whose positive root yields $b=\frac{2p_{k}a+ p_{k-1}}{p_{k+1}}$ and for $d$ to be an integer we have $a = p_{k+1}n+1$ and $b= 2p_{k}n+1$ where $p_k$ is the numerator of $c_k$ of $\sqrt{2}$.
\end{proof}

\section{Formulae with central term $=a_0$}

\begin{proposition}
Let $a_n$ be the central term of the period of the scf of $\sqrt{d} = \sqrt{a^2+b}$. Then $ a_n=a_0 \Longleftrightarrow b \equiv 2 \pmod 4.$
\end{proposition}
This implies that $b$ is always even, and $d \equiv a \pmod2$; $d$ is odd(even) if and only if $a$ is odd(even).

\subsection{Formulae with $\ell (d)=6$}

Kraitchik \cite[p.41]{kraitchik} gives two particular formulas ($m=1, 2$):
\begin{gather*}
\sqrt{(3n+1)^2 +(4n+2)} =[3n+1; ~ \overline{1, 2, 3n+1, 2, 1, 6n+2}],\\
\sqrt{(9n+2)^2 +(8n+2)} =[9n+2; ~ \overline{2, 4, 9n+2, 4, 2, 18n+4}].
\end{gather*}

\subsection{Formulae with $\ell (d)=8$}

We find in Kraitchik's book \cite[p.47]{kraitchik}:
\[
\sqrt{(7n+5)^2 +(8n+6)} =[7n+5; ~ \overline{1, 1, 3, 7n+5, 3, 1, 1, 14n+10}], \, n \in \mathbb{N}_0.
\]
It is a special case $(m=2)$ of a {\bf general formula} for any fixed $m \in \mathbb{N}\setminus \lbrace 1 \rbrace$:
\begin{align}
& \sqrt{\lbrace(2m^2 -1)n +2m^2 -m-1 \rbrace^2 +(4mn+ 4m -2)}\notag\\
& = \left[(2m^2 -1)n +2m^2 -m-1; \right.\notag \\
& \left. \overline{  m-1, 1, 2m-1, (2m^2 -1)n +2m^2 -m-1, 2m-1,} \right. \notag \\
& \left. \overline {1, m-1, 2\lbrace(2m^2 -1)n +2m^2 -m-1 } \right]. \label{amrik8a}
\end{align}
It can be proved by using  the continued fraction algorithm. The next special case $(m=3)$, true for $n=0, 1, 2, \dots$, follows.
\[
\sqrt{(17n+14)^2 +(12n+10)} =[17n+14; ~ \overline{2, 1, 5, 17n+14, 5, 1, 2, 34n+28}].
\]

\subsection{Formula with $\ell (d)=10$}

This formula ($n=1, 2, 3, \dots$) gives a combination of 1 and 2:
\begin{equation}
\sqrt{(11n+1)^2 +2(8n+1)} =[11n+1; ~ \overline{1, 2, 1, 2, 11n+1, 2, 1, 2, 1, 22n+2}]. \label{amrik10a}
\end{equation}
We also have a more general formula ($n=1, 2, 3, \dots$):
\begin{equation}
\sqrt{(9n+3)^2 +18} =[9n+3; ~ \overline{n, 2, 1, 2n, 9n+3, 2n, 1, 2, n, 18n+6}]. \label{amrik10b}
\end{equation}
I discovered an akin formula with only four fixed quotients:
\begin{equation}
\sqrt{(9n+6)^2 +18} =[9n+6; ~ \overline{n, 1, 2, 2n+1, 9n+6, 2n+1, 2, 1, n, 18n+12}]. \label{amrik10c}
\end{equation}
We also have these formulas ($n=0, 1, 2, \dots$)  
\begin{align}
& \sqrt{(17n +15)^2 +(20n +18) } \notag =[17n +15; \\
& \overline{1, 1, 2, 3, 17n +15, 3, 2, 1, 1, 2(17n +15)}].
\end{align}

\begin{align}
& \sqrt{(27n+18)^2 +(44n+30)} =[27n+18; \notag\\
& \overline{1, 4, 2, 2, 27n+18, 2, 2, 4, 1, 54n+36}]. \label{amrik10d}
\end{align}

\begin{align}
& \sqrt{(33n +21)^2 +(28n +18) } \notag =[33n +21; \\
& \overline{2, 2, 1, 4, 33n +21, 4, 1, 2, 2, 2(33n +21)}].
\end{align}

\begin{align}
& \sqrt{(321n +21)^2 +(152n +10) } \notag =[321n +21; \\
& \overline{4, 4, 2, 8, 321n +21, 8, 2, 4, 4, 2(321n +21)}].
\end{align}
All these formulas can be proved easily by means of the continued fraction algorithm. 

\subsection{Formula with $\ell (d)=12$}

We find this formula in \cite[p.51]{kraitchik}:
\begin{align}
& \sqrt{(23n+6)^2 +2 (18n+5)} =[23n+6; \notag \\
& \overline{1, 3, 1, 1, 2, 23n+6, 2, 1, 1, 3, 1, 46n+12}]; \quad n \in \mathbb{N}_0.
\end{align}
I obtained this formula:
\begin{align}
& \sqrt{(193n+11)^2 +2 (52n+3)} =[193n+11; \notag \\
& \overline{3, 1, 2, 2, 7, 193n+11, 7, 2, 2, 1, 3, 2(193n+11)}]; \quad n \in \mathbb{N}_0.
\end{align}

\subsubsection{Formula with $\ell (d)=14$}
We further have for $n=0, 1, 2, \dots$,
\begin{align}
& \sqrt{(153 n + 13)^2 + ( 2 (58 n + 5))} \notag\\
&=[153 n + 13; ~ \overline{2, 1, 1, 1, 3, 5, 153 n + 13, 5, 3, 1, 1, 1, 2, 2(153 n + 13)}].
\end{align}

\subsubsection{Formula with $\ell (d)=16$}
Furthermore, we have for $n=0, 1, 2, \dots$,
\begin{align}
& \sqrt{(217n +13)^2 + (356n +22) } \notag\\
&=[217n +13; ~ \overline{1, 4, 1, 1, 3, 2, 2, 217n +13, 2, 2, 3, 1, 1, 4, 1, 2(217n +13)}].
\end{align}

\subsubsection{Formula with $\ell (d)=18$}
Furthermore, we have for $n=0, 1, 2, \dots$,
\begin{align}
& \sqrt{(747n +11)^2 +2 (590 n + 9) } \notag\\
&=[747n +11; ~ \overline{1, 3, 1, 3, 7, 1, 1, 2, 747n +11, 2, 1, 1, 7, 3, 1, 3, 1, 2(747n +11)}].
\end{align}

\begin{align}
& \sqrt{(339n +137)^2 +(400n +162) } \notag =[339n +137; \\
& \overline{1, 1, 2, 3, 1, 1, 2, 3, 339n +137, 3, 2, 1, 1, 3, 2, 1, 1, 2(339n +137)}].
\end{align}

\begin{align}
& \sqrt{(1187n +850)^2 +(1008n +722) } \notag =[1187n +850; \\
& \overline{2, 2, 1, 4, 2, 2, 1, 4, 1187n +850, 4, 1, 2, 2, 4, 1, 2, 2, 2(1187n +850)}].
\end{align}

\subsubsection{Formula with $\ell (d)=20$}
Furthermore, we have for $n=0, 1, 2, \dots$,
\begin{align}
& \sqrt{(7199n +19)^2 +(2264n +6)} \notag\\
&=[7199n +19; ~ \overline{6, 2, 1, 3, 1, 1, 2, 1, 12, 7199n +19, 12, 1, 2, 1, 1, 3, 1, 2, 6, 2(7199n +19)}].
\end{align}

\subsubsection{Formula with $\ell (d)=22$}
Furthermore, we have for $n=0, 1, 2, \dots$,
\begin{align}
& \sqrt{(3201n +12)^2 +(5660n +22) } \notag\\
&=[3201n +12; ~ \overline{1, 7, 1, 1, 1, 2, 4, 1, 3, 2, 3201n +12, 2, 3, 1, 4, 2, 1, 1, 1, 7, 1, 2(3201n +12)}].
\end{align}

\subsubsection{Formula with $\ell (d)=26$}
Furthermore, we have for $n=0, 1, 2, \dots$,
\begin{align}
& \sqrt{(6763n +1354)^2 +(7980n +1598) } \notag =[6763n +1354; \\
& \overline{1, 1, 2, 3, 1, 1, 2, 3, 1, 1, 2, 3, 6763n +1354, 3, 2, 1, 1, 3, 2, 1, 1, 3, 2, 1, 1, 2(6763n +1354)}].
\end{align}

\begin{align}
& \sqrt{(42699n +6102)^2 +(36260n +5182) } \notag =[42699n +6102; \\
& \overline{2, 2, 1, 4, 2, 2, 1, 4, 2, 2, 1, 4, 42699n +6102, 4, 1, 2, 2, 4, 1, 2, 2, 4, 1, 2, 2, 2(42699n +6102)}].
\end{align}

\section{Formulas with replicating pair $\lbrace m, \, 2m \rbrace$, central term $=a_0$}
Define two sequences $A_{k+1} =4 A_k -A_{k-1}, \; B_{k+1} =4 B_k -B_{k-1}$ with $A_1=3, B_1 =1, A_0=1, B_0 =0.$ Then
\begin{theorem}\label{thm2}
\begin{align*}
& \sqrt{(A_{k}~ n +1)^2 +2(2B_{k} ~n +1)} = \left[A_{k} n +1; \right. \\
& \left. \overline{\text{1 \& 2 repeated k times}, (A_{k} n +1), \text{2 \& 1 repeated k times}, 2(A_{k} n +1)} \right], \notag
\end{align*}
\end{theorem}

\begin{proof}
The sequences defined here give $A_k/B_k$:
\[
\frac{3}{1}, \, \frac{11}{4}, \, \frac{41}{15}, \, \frac{153}{56}, \, \frac{571}{209}, \dots
\]
And let us now recall Lemma \ref{lemma2}. The first five powers of the matrix are:
\[
\begin{bmatrix}
3 & 2\\
1 & 1
\end{bmatrix}^2
=\begin{bmatrix}
11 & 8\\
4 & 3
\end{bmatrix};
\quad
\begin{bmatrix}
3 & 2\\
1 & 1
\end{bmatrix}^3
=\begin{bmatrix}
41 & 30\\
15 & 11
\end{bmatrix}
\]
\[
\begin{bmatrix}
3 & 2\\
1 & 1
\end{bmatrix}^4
=\begin{bmatrix}
153 & 112\\
56 & 41
\end{bmatrix};
\quad
\begin{bmatrix}
3 & 2\\
1 & 1
\end{bmatrix}^5
=\begin{bmatrix}
571 & 418\\
209 & 153
\end{bmatrix}.
\]
Thus $A_k =q_{2k} =p'_{2k-1}$ and $B_k =q_{2k-1} =q'_{2k-1}$ in \eqref{eq:l2} in Lemma \ref{lemma2}.

Applying the procedure followed in the proof of earlier theorem, we have $\sqrt{d} =[a; \overline {1, 2, a, 2, 1, 2a}]$:
\[
y = \cfrac{1}{1+ \cfrac{1}{2+ \cfrac{1}{a +\cfrac{1} {2 +\cfrac{1}{1+ \cfrac{1}{y+2a}}}}}}
\]
which after some long calculation leads to $3 y^2 + 6a y - (4a+2) =0.$ Solving the quadratic equation, we get the positive root: $y = -a +\sqrt{a^2 +\frac{4a+2}{3}}$. For $d$ to be an integer we have the solution $a=3n+1$ gives $b=4n+2$.

When $\sqrt{d} =[a; \overline {1, 2, 1, 2, a, 2, 1, 2, 1, 2a}]$, the same procedure gives the equation: $11y^2 +22ay -(16a+6) =0$ whose positive root yields $b=\frac{16a+6}{11}$ and for $d$ to be an integer we have $a=11n+1$ and $b=16n+2.$

With $\sqrt{d} =[a; \overline {1, 2, 1, 2, 1, 2, a, 2, 1, 2, 1, 2, 1, 2a}]$, the procedure gives the equation: $41y^2 +82ay -(60a+22) =0$ whose positive root yields $b=\frac{60a+22}{41}$ and for $d$ to be an integer we have $a=41n+1$ and $b=60n+2.$

With $\sqrt{d} =[a; \overline {1, 2, 1, 2, 1, 2, 1, 2, a, 2, 1, 2, 1, 2, 1, 2, 1, 2a}]$, we get the equation: $153y^2 +306y -(224a+82) =0$ whose positive root yields $b=\frac{224+82}{153}$ and for $d$ to be an integer we have $a=153n+1$ and $b=224n+2.$

And for $k$ times the pairs $(1, 2$ and $(2, 1)$, we immediately obtain the desired result by applying the formulas in \eqref{rip}.
\end{proof}

We now give a generalization of the preceding theorem.  
\begin{general}
Let $\sqrt{m^2 +2} =[m; \overline {m, 2m}]$ and $q_{2k}, \, q_{2k-1}$ be the denominators of its convergents; then the multiplier for convergents with gap 2 equals $M =2 (m^2 +1)$, and so:
$\displaystyle q_{2k+2} = M q_{2k} - q_{2k-2}; \quad q_{2k+1} = M q_{2k-1} - q_{2k-3}; \; q_0=1, \; q_{1} =m.
\displaystyle$
This gives $q_{2}= 2m^2 +1, \; q_{3}= 2m (m^2 +1), \; q_{4}= 2m^2 (2m^2 +3)+1$. We then have:
\begin{align*}
& \sqrt{(q_{2k}~ n +m)^2 +2(2q_{2k-1} ~n +1)} = \left[q_{2k}~ n +m; \right. \\
& \left. \overline{(m, 2m) ~\text{repeated  ~k ~times}, (q_{2k}~ n +m), (2m, ~m) ~\text{repeated} ~k ~\text{times}, 2(q_{2k}~ n +m)} \right].
\end{align*}
\end{general}

\subsection{Formulas for case $m=1$}
We noted some formulas for Theorem \ref{thm2}. 

\subsection{Formulas for case $m=2$}

The convergents of $\sqrt{6} =[2; \overline{2, 4}]$ are:
\[
\frac{5}{2}, \, \frac{22}{9}, \, \frac{49}{20}, \, \frac{218}{89}, \, \frac{485}{198}, \, \frac{2158}{881}, \, \frac{48271}{1960}, \, \frac{21362}{8721}, \dots
\]
They yield these formulas:

\begin{equation}
\sqrt{(9n+2)^2 +2(4n+1)} =\left[9n+2; \overline{2, 4, 9n+2, 4, 2, 2(9n+2)}\right].
\end{equation}

\begin{equation}
\sqrt{(89n+2)^2 +2(40n+1)} =\left[89n+2; \overline{2, 4, 2, 4, 89n+2, 4, 2, 4, 2, 2(89n+2)}\right].
\end{equation}

\begin{align}
& \sqrt{(881n+2)^2 +2(396n+1)} =\left[881n+2; \right. \notag \\
& \left. \overline{2, 4, 2, 4, 2, 4, 41n+1, 4, 2, 4, 2, 4, 2, 2(881n+2)}\right].
\end{align}

\subsection{Formulas for case $m=3$}

The convergents of $\sqrt{11} =[3; \overline{3, 6}]$ are:
\[
\frac{10}{3}, \, \frac{63}{19}, \, \frac{199}{60}, \, \frac{1257}{379}, \, \frac{3970}{1197}, \, \frac{25077}{7561}, \, \frac{79201}{23880}, \, \frac{500283}{150841}, \dots
\]
They lead to the following formulas.

\begin{equation}
\sqrt{(19n+3)^2 +2(2\cdot 3n+1)} =\left[19n+3; \overline{3, 6, 19n+3, 6, 3, 2(19n+3)}\right].
\end{equation}

\begin{equation}
\sqrt{(379n+3)^2 +2(2\cdot 60n+1)} =\left[379n+3; \overline{3, 6, 3, 6, 379n+3, 6, 3, 6, 3, 2(379n+3)}\right].
\end{equation}

\begin{align}
& \sqrt{(7561n+3)^2 +2(2\cdot 1197n+1)} =\left[7561n+3; \right. \notag \\
& \left. \overline{3, 6, 3, 6, 3, 6, 7561n+3, 6, 3, 6, 3, 6, 3, 2(7561n+3)}\right].
\end{align}

\subsection{Proof of the occurrence of $m, 2m$ once}

I now prove the formula wherein the pair $(m, 2m)$ occurs once ($k=1$) for any $m \in \mathbb{N}, \, n \in \mathbb{N}$:
\begin{align}
& \sqrt{\lbrace(2m^2 +1)n +m \rbrace^2 + 2(2mn+1)} =  \\
& [(2m^2 +1)n +m; ~ \overline{m, 2m, (2m^2 +1)n +m, 2m, m, 2(2m^2 +1)n +2m}]. \notag \label{amrik6k1}
\end{align}

\begin{proof}
\begin{align*}
& \sqrt {\lbrace(2m^2 +1) n +m \rbrace^2 +(4mn+2) }  \\
& =( (2m^2 +1) n +m ) +  \sqrt {\lbrace (2m^2 +1) n +m \rbrace^2 +(4mn+2)} -( (2m^2 +1) n +m )  \\
& =( (2m^2 +1) n +m ) +  \frac {4mn+2}{\sqrt {\lbrace (2m^2 +1) n +m \rbrace^2 +(4mn+2)} +( (2m^2 +1) n +m )}.
\end{align*}

So $a_0 = (2m^2 +1) n +m$ and the expression on the extreme right on being inverted becomes
\[
 m +  \frac{\sqrt {\lbrace (2m^2 +1) n +m \rbrace^2 +(4mn+2)} -( (2m^2 -1) n +m )} {4mn+2}.
\]

So  $a_1 = m$ and the expression on the extreme right becomes
\[
 \frac {2mn+1}{\sqrt {\lbrace (2m^2 +1) n +m \rbrace^2 +(4mn+2)} +( (2m^2 -1) n +m )}
\]
which on being inverted becomes
\[
2m + \frac{\sqrt {\lbrace (2m^2 +1) n +m \rbrace^2 +(4mn+2)} -( (2m^2+1) n +m )}  {2mn+1}.
\]

So  $a_2 = 2m$ and the expression on the extreme right becomes
\[
 \frac {2}{\sqrt {\lbrace (2m^2 +1) n +m \rbrace^2 +(4mn+2)} +( (2m^2 +1) n +m )}
\]
which on being inverted becomes
\[
( (2m^2 +1) n +m ) + \frac{\sqrt {\lbrace (2m^2 +1) n +m \rbrace^2 +(4mn+2)} -( (2m^2+1) n +m )}  {2}.
\]

So $a_3 = (2m^2 +1) n +m$ and the expression on the extreme right becomes

\[
 \frac {2mn+1}{\sqrt {\lbrace (2m^2 +1) n +m \rbrace^2 +(4mn+2)} +( (2m^2 +1) n +m )}
\]
which on being inverted becomes
\[
2m + \frac{\sqrt {\lbrace (2m^2 +1) n +m \rbrace^2 +(4mn+2)} -( (2m^2 -1) n +m )}  {2mn+1}.
\]

So  $a_4 = 2m$ and the expression on the extreme right becomes
\[
 \frac {4mn+2}{\sqrt {\lbrace (2m^2 +1) n +m \rbrace^2 +(4mn+2)} +( (2m^2 -1) n +m )}
\]
which on being inverted becomes
\[
 m +  \frac{\sqrt {\lbrace (2m^2 +1) n +m \rbrace^2 +(4mn+2)} -( (2m^2 +1) n +m )} {4mn+2}.
\]

So $a_0=m$ and  the expression on the extreme right becomes
\begin{align*}
&  \frac {1}{\sqrt {\lbrace (2m^2 +1) n +m \rbrace^2 +(4mn+2)} +( (2m^2 +1) n +m )} \\
& =\frac {1}{ 2 (( (2m^2 +1) n +m )) +\sqrt {\lbrace (2m^2 +1) n +m \rbrace^2 +(4mn+2)} -( (2m^2 +1) n +m )}
\end{align*}

Hence, $a_6 =2 (( (2m^2 +1) n +m )$ and the loop begins all over again.
\end{proof}

\subsection{Proof of the occurrence of $m, 2m$ twice}

Similarly, the formula wherein the pair $(m, 2m)$ occurs twice($k=2$) for any $m \in \mathbb{N}, \, n \in \mathbb{N}$:
\begin{align}
& \sqrt{ \lbrace (2m^2 (2m^2 +3) +1)n +m \rbrace^2 + 2(2m (m^2 +1)n +1)} = \notag \\
& \left[(2m^2 (2m^2 +3) +1)n +m; ~ \overline{ m, 2m, m, 2m, (2m^2 (2m^2 +3) +1)n +m}, \right. \notag\\
& \left. \overline{ 2m, m, 2m, m, 2((2m^2 (2m^2 +3) +1)n +m) } \right].  \label{amrik8k2}
\end{align}
can also be proved by the algorithm.

The general formula can be proved using the procedure used earlier though the proof is somewhat tedious.

\section{Formulas with repeated triple $\lbrace1, 1, 3\rbrace$ }

Define the sequences $\lbrace p_k\rbrace$, $\lbrace q_k\rbrace$:
\begin{gather}
p_{-1} =-1, \quad p_{0} =1, \quad p_{k} = 8~ p_{k-1} + p_{k-2}\; \text{for} \; k\ge 1, \\
q_{-1} =4, \quad q_{0} =0, \quad q_{k} = 8~ q_{k-1} + q_{k-2}\;  \text{for} \; k\ge 1.
\end{gather}
Its associated matrix is
\begin{equation}
\begin{bmatrix}
7 & 2\\
4 & 1
\end{bmatrix}^k
=
\begin{bmatrix}
p_{k} & q_{k}/2\\
q_{k} & p_{k-1}+q_{k-1}/2
\end{bmatrix}.
\end{equation}

We then have:
\begin{theorem}
\begin{align*}
& \sqrt{\left\lbrace p_k ~ n + \frac{p_k +3}{2}\right\rbrace^2 + \left(2q_{k} ~n +q_{k} +2\right)}  =
\left[p_{k}~ n +\frac{p_k +3}{2}; \overline{(1, 1, 3)~ k ~\text{times}}, \right. \\
& \left. \overline {p_{k}~ n +\frac{p_k +3}{2}, (3, 1, 1) ~\text{repeated} ~k ~\text{times}, 2\left(p_{k}~ n +\frac{p_k +3}{2}\right)} \right].
\end{align*}
\end{theorem}
\begin{proof}
This can be proved by mimicking the proof of Theorems 1 and 2.   
\end{proof}
The first three formulas yielded by this theorem follows.

\begin{equation}
\sqrt{(7n+5)^2 +(8n+6)} =[7n+5; ~ \overline{1, 1, 3, 7n+5, 3, 1, 1, 2(7n+5)}].
\end{equation}

\begin{align}
& \sqrt{(57n+30)^2 +(64n+34)} =\left [57n+30; \right. \notag \\
& \left. \overline{1, 1, 3, 1, 1, 3, 57n+30, 3, 1, 1, 3, 1, 1, 2(57n+30)} \right].
\end{align}

\begin{align}
& \sqrt{(463n+233)^2 +(520n+262)} =\left[463n+233; \right. \notag \\
& \left. \overline{1, 1, 3, 1, 1, 3, 1, 1, 3, 463n+233, 3, 1, 1, 3, 1, 1, 3, 1, 1, 2(463n+233)}\right].
\end{align}

\section{Repeated odd partial quotients}

\subsection{General formula for $\ell -1$ quotient $(2m+1), \, m\ge 0$}

\subsubsection {Methodology}

As derived earlier, $\frac{(2m+1) +\sqrt{(2m+1)^2 +4}}{2} =[2m+1; \overline{2m+1}]$. Let $(2m+1)$ be the repeated partial quotient in the period of the scf for $-(2m+1) +\sqrt{(2m+1)^2 +4}=[0; m, 1, 1, m, 4m+2]$ using the formula \eqref{amrik5a}. Now if we calculate the inverse of its convergent at each quotient we find that the denominator is odd at the quotient $a_1,$ at $a_2$ and at $a_4$ while it is even at the quotient $a_3$ and at $a_5.$ That is, only in the truncated fractions $(a_1, a_2, a_3)$ and $(a_1, a_2, a_3, a_4, 2a_0)$ the denominator is even. The \emph{trial and error method} revealed that the truncations with even denominators yield the desired formulas but not those having odd denominators. We normally use the notation $c_k =\frac{p_k}{q_k}$; I will use capital letters, to avoid confusion.

Taking the truncation $(a_1, a_2, a_3) =(m, 1, 1)$, we get the convergent $c_3 = \frac{p_3}{q_3}$ which we denote by $C_1=\frac{P_1}{Q_1}$:
\begin{equation}
C_1=\frac{P_1}{Q_1}  =\frac{2m+1}{2} = \frac{2m+1}{2}, \label{amrik3}
\end{equation}
and its successor with truncation after five more quotients $c_8 = \frac{p_8}{q_8} =(m, 1, 1, m, 4m+2, m, 1, 1)$ which we denote by $C_2=\frac{P_2}{Q_2}$:
\begin{equation}
C_2=\frac{P_2}{Q_2} =\frac{(2m+1)^4 +3(2m+1)^2 +1}{2(2m+1)^3 +4(2m+1)}. \label{amrik3a}
\end{equation}

Next we take the full period: $(a_1, a_2, a_3, a_4, a_5) =(m, 1, 1, m, 4m+2)$ and get the convergent $c_5 = \frac{p_5}{q_5}$
and denote it by $C^\prime_1=\frac{P^\prime_1}{Q^\prime_1}$:
\begin{equation}
C^\prime_1=\frac{P^\prime_1}{Q^\prime_1} =(m, 1, 1, m, 4m+2)=\frac{(2m+1) [(2m+1)^2 +2]}{2[(2m+1)^2 +1]}, \label{amrik4}
\end{equation}
and its successor with two full periods: $(m, 1, 1, m, 4m+2, m, 1, 1, m, 4m+2)$ and get the convergent $c_{10} = \frac{p_{10}}{q_{10}}$
and denote it by $C^\prime_2=\frac{P^\prime_2}{Q^\prime_2}$::
\begin{equation}
C^\prime_2 = \frac{P^\prime_2}{Q^\prime_2}  = \frac{(2m+1)^6 + 5(2m+1)^4 +6(2m+1)^2 +1}{2[(2m+1)^5 + 4(2m+1)^3 + 3(2m+1)}. \label{amrik4a}
\end{equation}

From these relations we deduce the the multiplier ($M$) for further convergents after every five quotients is: $M = (2m+1) [(2m+1)^2 +3]$. Then
\begin{equation}
P_{k+1} =M P_k +P_{k-1}; \; Q_{k+1} =M Q_k +Q_{k-1}, \quad k\ge 2. \label{amrikm1}
\end{equation}

\begin{equation}
P^\prime_{k+1} =M P^\prime_k +P^\prime_{k-1}; \quad Q^\prime_{k+1} =M Q^\prime_k +Q^\prime_{k-1}, \quad k\ge 2. \label{amrikm2}
\end{equation}

I have defined here $C_k: = c_{5k-2}$ and $P_k:= p_{5k-2}, Q_k: = q_{5k-2}$ and $C^\prime_k := c_{5k}$ and $P^\prime_k := p_{5k}, Q^\prime_k =: q_{5k}$, in terms of the regular convergents.

\subsubsection{General formula}

\begin{general}
(i) Let $P_k$ and $Q_k$ be the numbers as defined above. Then
\begin{align*}
& \sqrt{\left(P_{k} ~n ~ - \frac{P_{k} -(2m+1)}{2}\right)^2 +\left(Q_{k} ~n ~- \frac{Q_{k} -2}{2}\right)} \\
&= \left[P_{k}~ n~ - \frac{P_{k} -(2m+1)}{2}; \right. \\
~&  \left.  \overline{(2m+1) ~\text{repeated}~ (3k-2) ~\text{times}, ~2 \left(P_{k}~ n~ -\frac{P_{k} -(2m+1)}{2}\right)}\right].
\end{align*}
(ii) Let $P^\prime_k$ and $Q^\prime_k$ be the numbers as defined above. We then have
\begin{align*}
& \sqrt{\left(P^\prime_{k} ~n ~ - \frac{P^\prime_{k} -(2m+1)}{2}\right)^2 +\left(Q^\prime_{k} ~n ~- \frac{Q^\prime_{k} -2}{2}\right)} \\
&= \left[P^\prime_{k}~ n~ - \frac{P^\prime_{k} -(2m+1)}{2}; \right. \\
~&  \left.  \overline{(2m+1) ~\text{repeated }~3k ~\text{times}, ~2 \left(P^\prime_{k}~ n~ -\frac{P^\prime_{k} -(2m+1)}{2}\right)}\right].
\end{align*}
\end{general}

\begin{proof}
We have here the partial quotient $(2m+1)$ repeated $(3k-2)$ times.  In the $2\times2$ matrix approach, this just corresponds to the matrix
$\begin{bmatrix}
2m+1 & 1\\
1 & 0
\end{bmatrix}$
raised to the $3k-2$ power.  Now we just need to pre- and post-multiply by the $2\times2$ matrices corresponding to the first and last
partial quotient, and work out the equation. It would match the LHS of our General Theorem 2.

From our discussion in subsection \ref{genmat}, we get:
\[
\begin{bmatrix}
2m+1 & 1\\
1 & 0
\end{bmatrix}^{3k-2}
=
\begin{bmatrix}
u_{3k-1} & u_{3k-2}\\
u_{3k-2} & u_{3k-3}
\end{bmatrix};
\;
\begin{bmatrix}
2m+1 & 1\\
1 & 0
\end{bmatrix}^{3k}
=
\begin{bmatrix}
u_{3k+1} & u_{3k}\\
u_{3k} & u_{3k-1}
\end{bmatrix}
\]
where $u_1=1, u_2=2m+1.$ As noted earlier, $u_{3k}$ is always even while $u_{3k+1}$ and $u_{3k+2}$ are always odd for $k=0, 1, 2, \dots$ The pair $(u_{3k-1}, 2u_{3k-2})$ corresponds to $(P_{k}, Q_{k})$ and yields the quotient $(2m+1)$ repeated $(3k-2)$ times while $(u_{3k+1}, 2u_{3k})$ corresponds to $(P^\prime_{k}, Q^\prime_{k})$ and yields the quotient $(2m+1)$ repeated $(3k)$ times in the period. The pair $(u_{3k}, 2u_{3k-1})$, with $u_{3k}$ even, is inadmissible as the two numbers are not relatively prime because 2 divides both. Let us, for example, take $m=1$. We then have:
\[
\begin{bmatrix}
3 & 1\\
1 & 0
\end{bmatrix}^{3}
=
\begin{bmatrix}
33 & 10\\
10 & 3
\end{bmatrix};
\quad
\begin{bmatrix}
3 & 1\\
1 & 0
\end{bmatrix}^{4}
=
\begin{bmatrix}
109 & 33\\
33 & 10
\end{bmatrix}
\]

\[
\begin{bmatrix}
3 & 1\\
1 & 0
\end{bmatrix}^{6}
=
\begin{bmatrix}
1189 & 360\\
360 & 109
\end{bmatrix};
\quad
\begin{bmatrix}
3 & 1\\
1 & 0
\end{bmatrix}^{7}
=
\begin{bmatrix}
3927 & 1189\\
1189 & 360
\end{bmatrix}
\]
We will see soon that these are exactly the numbers and the associated formulas; the only difference is that all three numbers, excepting the first element in the first row and the first column, are twice those given here. Applying \eqref{rip} completes the proof.
\end{proof}

\subsection{Period with $\ell -1$ units, case $m=0$}

(i). Let $F_{3k-1}$ be the $(3k-1)$th Fibonacci number. Then for $k\in\mathbb{N}$ and $n=1, 2, \dots$
\begin{align*}
& \sqrt{\left(F_{3k-1}~n -\frac{F_{3k-1} -1}{2} \right)^2 +2 \left(F_{3k-2}~n -\frac{F_{3k-2} -1}{2} \right)} \\
& = \left[F_{3k-1}~n -\frac{F_{3k-1} -1}{2}; \overline{1 ~\text{repeated} ~ (3k-2)~ \text{times}, ~ 2\left(F_{3k-1}~n -\frac{F_{3k-1} -1}{2} \right)} \right].
\end{align*}
(ii). Let $F_{3k+1}$ be the $(3k+1)$th Fibonacci number. Then
\begin{align*}
& \sqrt{\left(F_{3k+1}~n -\frac{F_{3k+1}-1}{2} \right)^2 +2 \left(F_{3k} ~n -\frac{F_{3k}}{2} \right) +1} \\
& = \left[F_{3k+1}~n -\frac{F_{3k+1} -1}{2}; \overline{1 ~\text{repeated} ~ 3k~ \text{times}, ~ 2\left(F_{3k+1}~n -\frac{F_{3kr+1} -1}{2} \right)} \right].
\end{align*}
This is based on the discussion in \cite{kraitchik} and gave a clue to the general result. We have these recurrence relations (with $M=4$):
\[
F_{3k+4} = 4 F_{3k+1} +F_{3k-2}; \; F_{3k+3} = 4 F_{3k} +F_{3k-3}; \; F_{3k+2} = 4 F_{3k-1} +F_{3k-4}.
\]

\begin{case}
\begin{gather*}
(i).~ \sqrt{n^2 +2n} =[n; \overline{1, 2n}], \, n \in \mathbb{N}.\\
(ii).~ \sqrt{(3n-1)^2 +2(2n-1) +1} =[3n-1; \overline{1, 1, 1, 2(3n-1)}], \, n \in \mathbb{N}.
\end{gather*}
\end{case}

\begin{case}
\begin{gather*}
(i).~ \sqrt{(5n-2)^2 +2(3n-1)} =[5n-2; \overline{1, 1, 1, 1, 2(5n-2)}], \, n \in \mathbb{N}.\\
(ii).~ \sqrt{(13n-6)^2 +2(8n-4) +1} =[13n-6; \overline{1, 1, 1, 1, 1, 1, 2(13n-6)}].
\end{gather*}
\end{case}

\begin{case}
(i).~
\begin{align*}
& \sqrt{(21n-10)^2 +2(13n-6)} \\
& =[21n-10; \overline{1, 1, 1, 1, 1, 1, 1, 2(21n-10)}].
\end{align*}
(ii).~
\begin{align*}
& \sqrt{(55n-27)^2 +2(34n-17)+1} \\
& =[55n-27; \overline{1, 1, 1, 1, 1, 1, 1, 1, 1, 2(55n-27)}].
\end{align*}
\end{case}

\subsection{Period with $\ell -1$ threes} Taking $m=1,$ we have
\[
\sqrt{13} =[3; \overline{1, 1, 1, 1, 6}]; \; -3 +\sqrt{13} =[0; \overline{1, 1, 1, 1, 6}].
\]
The relations derived earlier straightway yield $\frac{P_k}{Q_k}$ and $\frac{P^\prime_k}{Q^\prime_k}.$ We are also giving the associated terminating continued fractions for elucidation.

\begin{case}
(i). We have $\frac{P_1}{Q_1} = (1, 1, 1) =\frac{3}{2}.$ So
\[
\sqrt{(3n)^2 +2n} = \left[3n; \overline{3, 6n)} \right].
\]
(ii). Next, $\frac{P^\prime_1}{Q^\prime_1} = (1, 1, 1, 1, 6) = \frac{33}{20}.$ So
\[
\sqrt{(33n-15)^2 +(20n-9)} =[33n-15; \overline{3, 3, 3, 2(33n-15)}].
\]
\end{case}

\begin{case}
(i). We have $\frac{P_2}{Q_2} =(1, 1, 1, 1, 6, 1, 1, 1)  = \frac{109}{66}.$ So
\[
\sqrt{(109n - 53)^2 +(66n -32)} = \left[109n-53; \overline{3, 3, 3, 3, 2(109n-53)} \right].
\]
(ii). Next, $\frac{P^\prime_2}{Q^\prime_2} = (1, 1, 1, 1, 6, 1, 1, 1, 1, 6) = \frac{1189}{720}.$ So
\begin{align*}
& \sqrt{(1189n -593)^2 +(720n -359)} \\
& =[1189n -593; \overline{3, 3, 3, 3, 3, 3, 2(1189n -593)}].
\end{align*}
\end{case}

\begin{case}
(i). We have $\frac{P_3}{Q_3} =(1, 1, 1, 1, 6, 1, 1, 1, 1, 6, 1, 1, 1) =\frac{36\cdot109 +3}{36\cdot66 +2} = \frac{3927}{2378}.$ So
\begin{align*}
& \sqrt{(3927n - 1962)^2 +(2378n -1188)} = \left[3927n - 1962; \right. \\
& \left. \overline{3, 3, 3, 3, 3, 3, 3, 2(3927n - 1962)} \right].
\end{align*}
(ii). Next, $\frac{P^\prime_3}{Q^\prime_3} = (1, 1, 1, 1, 6, 1, 1, 1, 1, 6, 1, 1, 1, 1, 6) =\frac{36\cdot1189 +33}{36\cdot720 +20} =\frac{42837}{25940}.$ So
\begin{align*}
&\sqrt{(42837n  -21417)^2 +(25940n-12969)}\\
& =[42837n  -21417; \overline{3, 3, 3, 3, 3, 3, 3, 3, 3, 2(42837n  -21417)}].
\end{align*}
\end{case}

\subsection{Period with $\ell -1$ fives} $m=2$ gives $-5 +\sqrt{29}$ $=[0; \overline{2, 1, 1, 2, 29}]$.
Using the general formula yields

\begin{case}
\[
(i). ~\sqrt{(5n)^2 +2n} = \left[5n; \overline{5, 10n)} \right].
\]

\[
(ii). ~ \sqrt{(135n-65)^2+(52n -25)} = \left[ 135n-65; \overline{5, 5, 5, 2(135n-65)} \right].
\]
\end{case}

\begin{case}
\[
(i). ~ \sqrt{(701n - 348)^2 +(270n -134)} = \left[701n - 348; \overline{5, 5, 5, 5, 2(701n - 348)} \right].
\]

\begin{align*}
& (ii). ~ \sqrt{(18901n -9448)^2 +(7280n-3639)} = \\
& \left[18901n -9448; \overline{5, 5, 5, 5, 5, 5, 2(18901n -9448)} \right].
\end{align*}
\end{case}

\begin{case}
\begin{align*}
& (i). ~\sqrt{(98145n - 49070)^2 +(37802n -18900)} = \left[98145n - 49070; \right. \\
& \left. \overline{5, 5, 5, 5, 5, 5, 5, 2(98145n - 49070)} \right].
\end{align*}

\begin{align*}
& (ii). ~\sqrt{ (2646275n -1323135)^2 +(1019252n-509625)} =\\
& \left[2646275n -1323135; \overline{5, 5, 5, 5, 5, 5, 5, 5, 5, 2(2646275n -1323135)} \right].
\end{align*}
\end{case}

One can go on like this and deduce formulas for larger values of $m$ by varying convergents.

\section{General formula for $\ell -1$ quotient $(2m), m\ge 1$}

\subsection{Methodology}

Let $(2m), \, m\ge 1$ be the quotient that repeats in the period. We then use the scf for $\sqrt{(2m)^2 +4}$, the special case ($n=2$) of \eqref{amrik1} for the needed convergents. I hit on this scf after examining the formulas given on pages 59-60 of Kraitchik's book \cite{kraitchik}.
The convergents $c_k =\displaystyle \frac{p_k}{q_k}$ can be computed with the recurrence relations $k\ge 1$:
\begin{gather*}
p_{2k} =m p_{2k-1} +p_{2k-2}; \; p_{2k+1} =4m p_{2k} +p_{2k-1}; \; p_0=1, \; p_1=2m;\\
q_{2k} =m q_{2k-1} +q_{2k-2}; \; q_{2k+1} =4m q_{2k} +q_{2k-1}; \; q_0 =0, \; q_1=1.
\end{gather*}

\subsection{General formula}

We then have this general formula for any fixed $m$:
\begin{general}
\begin{align*}
& \sqrt{(q_{k+1} n)^2 +(p_{k+1}) n +m^2 +1} \\
& = \left[q_{k+1} n~ +m; \overline{2m ~ \text{repeated} ~k ~\text{times}, 2 (q_{k+1}n +m)}\right]
\end{align*}
\end{general}

\begin{proof}
Here we have to raise the associated matrix by a power $k.$ We have
\begin{align*}
& c_1=\frac{p_1}{q_1} =\frac{2m}{1}; \; c_2 = (2m, m) =\frac{p_2}{q_2} =\frac{2m^2 +1}{m}; \\
& c_3 = (2m, m,  4m) =\frac{p_3}{q_3} =\frac {4 (2 m^2 + 1) m + 2 m}{4 m^2 + 1},
\end{align*}
which give us the recurrence relations that we used in our method.
\end{proof}

\subsection{Period with $\ell -1$ twos}
We have $\sqrt{2^2 +4} =[2; \overline{1, 4}].$ Its convergents $c_k =\displaystyle \frac{p_k}{q_k}$ can be computed with the recurrence relations $k\ge 1$:
\begin{gather*}
p_{2k} =p_{2k-1} +p_{2k-2}; \; p_{2k+1} =4p_{2k} +p_{2k-1};  \; p_0=1, \; p_1=2;\\
q_{2k} =q_{2k-1} +q_{2k-2}; \; q_{2k+1} =4q_{2k} +q_{2k-1}; \; q_0 =0, \; q_1=1.
\end{gather*}
The convergents (numerators/denominators in \url {https://oeis.org/A041010} and \url{https://oeis.org/A041011}) are:
\[
\frac{2}{1}, ~\frac{3}{1}, ~\frac{14}{5}, ~\frac{17}{6}, ~\frac{82}{29}, ~\frac{99}{35}, ~\frac{478}{169}, ~\frac{577}{204}, ~\frac{2786}{985}, \dots
\]
The first five cases follow (cf. \cite[59]{kraitchik}\cite[p.327]{sierpinski}):
\begin{gather*}
\sqrt{n^2 +3n +2} =\left[n+1; \overline{2, 2n+2} \right],\\
\sqrt{(5n)^2 +14n +2} =\left[5n+1; \overline{2, 2, 10n+2} \right],\\
\sqrt{(6n)^2 +17n +2} =\left[6n+1; \overline{2, 2, 2, 12n+2} \right],\\
\sqrt{(29n)^2 +82n +2} =\left[29n+1; \overline{2, 2, 2, 2, 58n+2} \right],\\
\sqrt{(35n)^2 +99n +2} =\left[35n+1; \overline{2, 2, 2, 2, 2, 70n+2} \right].
\end{gather*}
The matrix
$\begin{bmatrix}
3 & 8\\
1 & 3
\end{bmatrix}^k
=
\begin{bmatrix}
p_{2k} & 8 q_{2k}\\
q_{2k} & p_{2k}
\end{bmatrix}$
gives even convergents and $(2k-1)$ twos.

\begin{remark}
Sierpinski takes $2n, \, 12n, \, 70n,$ etc. in place of $n, \, 6n, 35n,$ etc. in the above formulas and so missed all odd $d$'s such as $\sqrt{55} = \left[7; \overline{2, 2, 2, 14} \right]$.
\end{remark}

\subsection{Period with $\ell -1$ fours}

We have $\sqrt{4^2 +4} =[4; \overline{2, 8}].$ The convergents are
\[
\frac{4}{1}, ~\frac{9}{2}, ~\frac{76}{17}, ~\frac{161}{36}, ~\frac{1364}{305}, ~\frac{2889}{646}, ~\frac{24476}{5473}, ~\frac{51841}{11592}, ~\frac{439204}{98209}, \dots
\]
The first five cases(cf. \cite[60]{kraitchik}) follow:
\begin{case}
\[
\sqrt{(2n)^2 + 9 n + 5} =[2n+2; \overline{4, 4n+4}].
\]
\end{case}

\begin{case}
\[
\sqrt{(17n)^2 +76n +5} =[17n+2; \overline{4, 4, 34n+4}].
\]
\end{case}

\begin{case}
\[
\sqrt{(36n)^2 +161n +5} =[36n+2; \overline{4, 4, 4, 72n+4}].
\]
\end{case}

\begin{case}
\[
\sqrt{(305n)^2 +1364n +5} =[305n+2; \overline{4, 4, 4, 4, 610n+4}].
\]
\end{case}

\begin{case}
\[
\sqrt{(646n)^2 +2889n +5} =[646n+2; \overline{4, 4, 4, 4, 4, 1292n+4}].
\]
\end{case}

\subsection{Period with $\ell -1$ sixes}

We have $\sqrt{6^2 +4} =[6; \overline{3, 12}].$ Its convergents are:
\[
\frac{6}{1}, ~\frac{19}{3}, ~\frac{234}{37}, ~\frac{721}{114}, ~\frac{8886}{1405}, ~\frac{27379}{4329}, ~\frac{33743}{53353}, \dots
\]
The first few cases are:

\begin{case}
\[
\sqrt{(3n)^2 +19n +10} =[3n+3; \overline{6, 6n+6}].
\]
\end{case}

\begin{case}
\[
\sqrt{(37n)^2 +234n +10} =[37n+3; \overline{6, 6, 74n+6}].
\]
\end{case}

\begin{case}
\[
\sqrt{(114n)^2 +721n +10} =[114n+3; \overline{6, 6, 6, 228n+6}].
\]
\end{case}

\begin{case}
\[
\sqrt{(1405n)^2 + 8886n +10} =[1405n+3; \overline{6, 6, 6, 6, 2810n+6}].
\]
\end{case}

One can go on like this indefinitely.

\section{Concluding remarks}

I have obtained some formulas containing certain patterns in the continued fractions of square roots. There is still enough scope to look for more patterns.

\section{Acknowledgments}
I would like to thank Jeffrey Shallit for suggesting proof using matrices, and Stephen Lucas for the derivation in 3.1.3.

\end{document}